\documentclass[11pt]{amsart}

\usepackage{graphicx}
\usepackage[USenglish,english]{babel}
\usepackage[T1]{fontenc}
\usepackage[latin1]{inputenc}
\usepackage{amsfonts}
\usepackage{amssymb}
\usepackage{amsthm}
\usepackage{amsmath}
\usepackage{tikz}
\usepackage{float}
\usepackage{enumerate}
\usepackage{accents}
\usepackage{mathtools}
\usepackage{mathrsfs}
\usepackage{comment}
\usepackage{afterpage}
\usepackage{bbm}
\usepackage{dsfont}
\usepackage{stmaryrd}
\usepackage{hyperref}
\usepackage{mleftright}
\usepackage{subfig}

\theoremstyle{plain}
\newtheorem{thm}{Theorem}[section]
\newtheorem{lem}[thm]{Lemma}
\newtheorem{prop}[thm]{Proposition}
\newtheorem{cor}[thm]{Corollary}
\newtheorem*{thm*}{Main Theorem}
\newtheorem*{prop*}{Proposition}
\newtheorem*{cor*}{Corollary}

\theoremstyle{definition}
\newtheorem{defn}[thm]{Definition}

\newtheorem{rmk}[thm]{Remark}
\newtheorem*{quest*}{Question}

\DeclareMathOperator{\op}{op} 

\renewcommand{\o}{\circ}

\newcommand{\R}{\mathbb{R}}
\newcommand{\Z}{\mathbb{Z}}
\newcommand{\N}{\mathbb{N}}
\renewcommand{\H}{\mathbb{H}}

\newcommand{\s}{\sigma}
\newcommand{\ra}{\rightarrow}
\newcommand{\Ra}{\Rightarrow}

\newcommand{\cu}{\subseteq}
\newcommand{\g}{\gamma}

\newcommand{\mbb}{\mathbb}
\newcommand{\mc}{\mathcal}
\newcommand{\mf}{\mathfrak}
\newcommand{\x}{\times}
\newcommand{\eps}{\epsilon}

\newcommand{\Aut}{\mathrm{Aut}}

\newcommand{\mscr}{\mathscr}

\newcommand{\crt}{\mathrm{crt}}
\newcommand{\Cr}{\mathrm{cr}}
\newcommand{\Sym}{\mathrm{Sym}}
\newcommand{\lk}{\mathrm{lk}}

\renewcommand{\ll}{\llbracket}
\newcommand{\rr}{\rrbracket}
\newcommand{\CAT}{{\rm CAT(0)}}
\newcommand{\CA}{{\rm CAT(-1)}}

\title[$\CAT$ cube complexes are determined by their boundary]{$\CAT$ cube complexes are determined by their boundary cross ratio}
\author{Jonas Beyrer, Elia Fioravanti and Merlin Incerti-Medici}

\begin{document}

\begin{abstract}
We introduce a $\Z$--valued cross ratio on Roller boundaries of $\CAT$ cube complexes. We motivate its relevance by showing that every cross-ratio preserving bijection of Roller boundaries uniquely extends to a cubical isomorphism. Our results are strikingly general and even apply to infinite dimensional, locally infinite cube complexes with trivial automorphism group.
\end{abstract}

\maketitle

\section{Introduction.}

Gromov boundaries of $\CA$ spaces are naturally endowed with a notion of \emph{cross ratio}. A classical example is provided by the standard projective cross ratio on $\partial_{\infty}\H^2\simeq\R\mbb{P}^1$. In the present paper, we introduce a similar object on the \emph{Roller boundary} of any $\CAT$ cube complex $X$ and show that this suffices to fully reconstruct the structure of $X$.

Our motivation essentially comes from two separate points of view. First, the theory of cube complexes has become fundamental within geometric group theory, proving extremely fruitful in relation to various questions stemming from low dimensional topology and group theory. The geometry of many interesting groups is encoded by a $\CAT$ cube complex, from classical examples such as right-angled Artin and Coxeter groups, to more recent discoveries like hyperbolic $3$--manifold or free-by-cyclic groups \cite{Bergeron-Wise,Kahn-Markovic,Hagen-Wise1,Hagen-Wise2}, to pathological situations such as Thompson's groups \cite{Farley} and Higman's group \cite{Martin}.

As a second perspective, boundary cross ratios provide a valuable tool in the study of negatively and non-positively curved spaces, often appearing in relation to strong rigidity results \cite{Otal-Ann,Bourdon-IHES,Bourdon-GAFA}. If $X$ is a Gromov hyperbolic or $\CAT$ space, its \emph{quasi-isometry} type is fully determined by a cross ratio, respectively, on the Gromov or contracting boundary \cite{Paulin,Mousley-Russell,Charney-Cordes-Murray}. By contrast, it is an open question whether these boundary cross ratios can be used to recover the \emph{isometry} type of $X$ (cf.\ \cite{Biswas1,Beyrer}). 

When $X$ is the universal cover of a closed, negatively-curved Riemannian manifold, the latter problem turns out to be essentially equivalent to the famous marked-length-spectrum rigidity conjecture (Problem~3.1 in \cite{Burns-Katok}), which has been fully solved only in dimension two \cite{Otal-Ann}.
   
In the present paper, we aim to overcome some of the issues arising in general $\CAT$ spaces by relying on the ``combinatorial structure'' available in cube complexes. This will allow us to define a simpler boundary cross ratio, which in fact only takes integer values. We will show that this cross ratio fully determines the $\CAT$ cube complex $X$ up to isometry and, in fact, even up to \emph{cellular} isometries --- ``\emph{cubical isomorphisms}'' in our terminology.

To this end, we observe that every $\CAT$ cube complex $X$ is endowed with two natural metrics: the $\CAT$ metric and the $\ell^1$ (or \emph{combinatorial}) metric. When $X$ is finite dimensional, these are bi-Lipschitz equivalent. We will restrict our attention to the $\ell^1$ metric --- which we denote by $d$ --- as this enables us to better exploit the cellular structure of $X$.

It is then natural to consider the horoboundary of the metric space $(X,d)$, usually known as \emph{Roller boundary} $\partial X$. This has by now become a standard tool in the study of cube complexes; see e.g.\ \cite{BCGNW,Nevo-Sageev,CFI,Fernos,Fernos-Lecureux-Matheus} for a (non-exhaustive) list of applications. We remark that, unlike Gromov and visual boundaries, Roller boundaries are always totally disconnected, as is reasonable to expect from objects associated to a cell complex.

By analogy with the $\CA$ context, it is reasonable to define a \emph{cross ratio}\footnote{For us, a \emph{cross ratio} is an $\R$--valued function defined on generic $4$--tuples of boundary points and satisfying certain symmetries (see (i)--(iv) in Section~\ref{cross ratio section}). This should be compared to analogous notions in \cite{Otal-IberoAm,Hamenstaedt-ErgodTh,Labourie-IM}.} in terms of \emph{Gromov products}\footnote{See Definition~III.H.1.19 in \cite{BH} or Section~\ref{cross ratio section} below for a definition.} in the metric space $(X,d)$. The result is a function $\Cr\colon\mscr{A}\ra\Z\cup\{\pm\infty\}$, defined on a subset $\mscr{A}\cu(\partial X)^4$. We show that this cross ratio admits the alternative expression:
\begin{align*}
\Cr(x,y,z,w)=\#\mscr{W}(x,z|y,w)-\#\mscr{W}(x,w|y,z),
\end{align*}
where $\mscr{W}(x,z|y,w)$ denotes the set of hyperplanes of $X$ that separate $x$ and $z$ from $y$ and $w$. In particular, it is clear that $\Cr$ is preserved by the diagonal action of $\Aut(X)$ on $(\partial X)^4$ and independent of any choices in its definition.

Generalising a result of \cite{Beyrer-Schroeder} for trees, we then prove the following.

\begin{thm*}
Let $X$ and $Y$ be $\CAT$ cube complexes with no extremal vertices and not isomorphic to $\R$. Every cross-ratio preserving bijection $f\colon\partial X\ra\partial Y$ uniquely extends to a cubical isomorphism $F\colon X\ra Y$. 
\end{thm*}

Note that the theorem does not require finite dimensional or locally finite cube complexes, nor any group action. The requirement that no vertex be extremal is necessary in order to prevent us from modifying a bounded portion of the cube complex without affecting the boundary; in the case of trees, this would amount to requiring that there are no leaves. 

We introduce extremal vertices in Definition~\ref{sharp corners defn}. Absence of extremal vertices can be viewed as an intermediate requirement between the geodesic extension property for the $\ell^1$ and $\CAT$ metrics\footnote{For complete cube complexes, it is well-known that the $\CAT$ metric satisfies the geodesic extension property if and only if $X$ has no free faces (Proposition~II.5.10 in \cite{BH}). Such spaces do not have extremal vertices (Remark~\ref{no extremal vertices}) and every cube complex without extremal vertices has the geodesic extension property w.r.t.\ the $\ell^1$ metric (Lemma~\ref{straight rays}).}. The Main Theorem holds more generally when every vertex satisfies Lemma~\ref{skinny points}.

In \cite{BF1,BF2}, the Main Theorem is extended to cross-ratio preserving bijections between much smaller subsets of the Roller boundaries. This has applications to length-spectrum rigidity questions for actions on cube complexes. The price to pay is that stronger assumptions need to be imposed on $X$ and $Y$. 

\medskip
We now briefly sketch the strategy of proof of the Main Theorem. To any three points $x,y,z\in\partial X$, we can associate two well-known objects: the \emph{interval} $I(x,y)\cu X\cup\partial X$ and the \emph{median} $m(x,y,z)\in I(x,y)$; see e.g.\ \cite{Nevo-Sageev}. These should be interpreted, respectively, as the union of all infinite geodesics between $x$ and $y$ and as a barycentre for the triangle $xyz$. It is a natural attempt to define the map $F:X\to Y$ as $F(v):=m(f(x),f(y),f(z))$, assuming for simplicity that there exist points $x,y,z\in\partial X$ with $v=m(x,y,z)$. 

However, even relying on the assumption that $f$ preserves cross ratios, it is a priori unclear whether $F$ is well-defined, i.e.\ independent of the choice of $x,y,z$. As an illustration of this, consider the two cube complexes $X$ and $Y$ pictured in in Figure~\ref{spanning a cube}. In both cases, the points $x,y,z,z'$ lie in the Roller boundary and satisfy $\Cr(x,y,z,z')=0$; the same holds if we permute the four points. In other words, cross ratios involving only $x$, $y$, $z$ and $z'$ cannot tell the two cases apart, even though we have $m(x,y,z)=m(x,y,z')$ in $X$ and $m(x,y,z)\neq m(x,y,z')$ in $Y$.

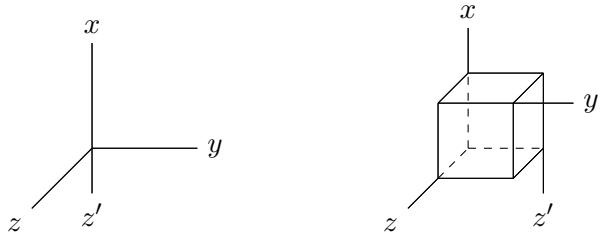
\begin{figure}
\begin{tikzpicture}
\begin{scope}
\draw [fill] (-0.4,-0.4) -- (0.4,0.4);
\draw [fill] (0.4,0.4) -- (0.4,1.8);
\draw [fill] (0.4,0.4) -- (1.8,0.4);
\draw [fill] (0.4,0.4) -- (0.4,-0.2);
\node [below left] at (-0.4, -0.4) {$z$};
\node [above] at (0.4,1.8) {$x$};
\node [right] at (1.8,0.4) {$y$};
\node [below] at (0.4,-0.2) {$z'$};
\end{scope}

\begin{scope}[xshift=5cm]
\draw [fill] (0,0) -- (0,1);
\draw [fill] (0,0) -- (1,0);
\draw [fill] (1,0) -- (1, 1);
\draw [fill] (0,1) -- (1, 1);
\draw [fill] (0,1) -- (0.4, 1.4);
\draw [fill] (1,0) -- (1.4, 0.4);
\draw [fill] (1,1) -- (1.4, 1.4);
\draw [fill] (0.4,1.4) -- (1.4,1.4);
\draw [dashed] (0,0) -- (0.4, 0.4);
\draw [dashed] (0.4,0.4) -- (0.4, 1.4);
\draw [dashed] (0.4,0.4) -- (1.4, 0.4);
\draw [fill] (1.4,0.4) -- (1.4, 1.4);
\draw [fill] (0,0) -- (-0.4, -0.4);
\draw [fill] (0.4,1.4) -- (0.4, 2);
\draw [fill] (1,1) -- (1.8, 1);
\draw [fill] (1.4,0.4) -- (1.4, -0.2);
\node [below left] at (-0.4, -0.4) {$z$};
\node [above] at (0.4,2) {$x$};
\node [right] at (1.8,1) {$y$};
\node [below] at (1.4,-0.2) {$z'$};
\end{scope}
\end{tikzpicture}
\caption{The cube complex $X$, on the left, is a tree with a single branch point and four boundary points $x$, $y$, $z$ and $z'$. Pictured on the right is a portion of $Y\simeq\R^3$; the points $x$, $y$, $z$ and $z'$ lie in the Roller boundary $\partial\R^3$.}
\label{spanning a cube} 
\end{figure}

We resolve the problem by only representing ${v=m(x,y,z)}$ with triples $(x,y,z)$ where $v$ disconnects the interval $I(x,y)$; in this case, we say that $x$ and $y$ are \emph{opposite} with respect to $v$. Examining Figure~\ref{spanning a cube}, it is easy to see that $x$ and $y$ are opposite in $X$, but not in $Y$. It can be shown that most vertices $v$ are of the form $v=m(x,y,z)$ for a triple $(x,y,z)$ such that $x$ and $y$ are opposite (Lemma~\ref{skinny points}) and, moreover, such triples can be characterised in terms of cross ratios (Lemma~\ref{op vs lcrt}).

\medskip
We conclude the introduction by remarking that the Main Theorem does not generalise to cross-ratio preserving \emph{embeddings} $\partial X\hookrightarrow \partial Y$. This stands in contrast with the behaviour of trees \cite{Beyrer-Schroeder} and rank-one symmetric spaces \cite{Bourdon-IHES}. A simple counterexample is provided by the cube complexes in Figure~\ref{spanning a cube} and the map $\partial X\hookrightarrow\partial Y=\partial\R^3$ that pairs points of the same name. 

For a counterexample involving \emph{cocompact} spaces, the above can be ad\-apt\-ed as follows. Let $X$ be the $4$--regular tree $T_4$ with each edge divided into three edges of length $1$. Let $Y$ be the $4$--regular tree $T_4$ with each vertex blown up to a $3$--cube as in Figure~\ref{spanning a cube}; thus, for every vertex of $T_4$, there is a $3$--cube in $Y$ and, for every edge of $T_4$, there is an edge of $Y$ joining two cubes. It is not hard to check that the natural homeomorphism $\partial X\ra\partial Y$ is cross-ratio preserving. Now, if we denote by $S$ the universal cover of the Salvetti complex of $\Z^3\ast\Z$, we can embed isometrically $Y\hookrightarrow S$ as a convex subcomplex. This gives rise to a cross-ratio preserving embedding $\partial X\simeq\partial Y\hookrightarrow\partial S$, which does not extend to an isometric embedding $X\hookrightarrow S$.



\medskip
{\bf Acknowledgements.} We thank Cornelia Dru\c{t}u, Ruth Charney, Tobias Hartnick and Viktor Schroeder for helpful conversations. We also express our gratitude to the organisers of the following conferences, where part of this work was carried out: \emph{Third GEAR (Junior) Retreat}, Stanford, 2017; \emph{Moduli Spaces}, Ventotene, 2017; \emph{Young Geometric Group Theory VII}, Les Diablerets, 2018; \emph{Topological and Homological Methods in Group Theory}, Bielefeld, 2018; \emph{3--Manifolds and Geometric Group Theory}, Luminy, 2018; \emph{Representation varieties and geometric structures in low dimensions}, Warwick, 2018. EF also thanks Viktor Schroeder for a visit at UZH. 

JB and MI-M acknowledge support by the Swiss National Science Foundation under Grant~200020$\setminus$175567. EF was supported by the Clarendon Fund and the Merton Moussouris Scholarship.

\section{Preliminaries.}

\subsection{$\CAT$ cube complexes.}

For an introduction to $\CAT$ cube complexes, we refer the reader to \cite{Sageev-notes}. In this subsection, we only recall some of the relevant terminology.

Let $X$ be a simply connected cube complex satisfying Gromov's no-$\triangle$-condition (see 4.2.C in \cite{Gromov-HypGps} and Chapter~II.5 in \cite{BH}). The Euclidean metrics on its cubes fit together to yield a $\CAT$ metric on $X$. In the present paper, however, we prefer to endow $X$ with its \emph{combinatorial metric}. More precisely, if $v,w\in X$ are vertices, $d(v,w)$ is the defined as the minimal length of a path connecting $v$ and $w$ within the $1$--skeleton of $X$. When $X$ is finite dimensional, the $\CAT$ and combinatorial metrics are bi-Lipschitz equivalent and complete.

Unless specified otherwise, {\bf all points $v\in X$ are implicitly understood to be vertices}; we do not distinguish between $X$ and its $0$--skeleton. Throughout the paper, the letter $d$ denotes the combinatorial metric on $X$. All geodesics are meant with respect to the combinatorial metric $d$; in particular, they are sequences of edges. We will nevertheless refer to $X$ by the more familiar expression `\emph{$\CAT$ cube complex}'. 

For every vertex $v\in X$, we define a graph $\lk(v)$. Its vertices are the edges of $X$ incident to $v$; vertices of $\lk(v)$ are joined by an edge if and only if the corresponding edges of $X$ span a square. We refer to $\lk(v)$ as the \emph{link} of $v$.

Let $\mscr{W}(X)$ and $\mscr{H}(X)$ be, respectively, the sets of hyperplanes and halfspaces of $X$. We simply write $\mscr{W}$ and $\mscr{H}$ when there is no need to specify the cube complex. We denote by $\mf{h}^*$ the complement of the halfspace $\mf{h}$. 

Two distinct hyperplanes are \emph{transverse} if they cross. If $e\cu X$ is an edge, we write $\mf{w}(e)$ for the hyperplane dual to $e$. We say that a hyperplane $\mf{w}$ is \emph{adjacent} to a point $v\in X$ if $\mf{w}=\mf{w}(e)$ for an edge $e$ incident to $v$.

We will generally confuse geodesics and their images as subsets of $X$. If $\g\cu X$ is an (oriented) geodesic, we denote by $\g(0)$ its initial vertex and by $\g(n)$ its $n$-th vertex. We refer to bi-infinite geodesics simply as \emph{lines}.

Every geodesic $\g\cu X$ can be viewed as a collection of edges; distinct edges $e,e'\cu\g$ must yield distinct hyperplanes $\mf{w}(e)$ and $\mf{w}(e')$. We write $\mscr{W}(\g)$ for the collection of hyperplanes crossed by (the edges of) $\g$. If two geodesics $\g$ and $\g'$ share an endpoint $v\in X$, their union $\g\cup\g'$ is again a geodesic if and only if $\mscr{W}(\g)\cap\mscr{W}(\g')=\emptyset$. 

\begin{lem}\label{straight rays 2} 
Given $v\in X$ and rays $r_1,r_2\cu X$ based at $v$, let $\mscr{W}_i\cu\mscr{W}(r_i)$ denote the subset of hyperplanes adjacent to $v$. The union $r_1\cup r_2$ is a line if and only if $\mscr{W}_1\cap\mscr{W}_2=\emptyset$.
\end{lem}
\begin{proof}
If $r_1\cup r_2$ is not a geodesic, there exists $\mf{w}\in\mscr{W}(r_1)\cap\mscr{W}(r_2)$. If $\mf{w}$ is not adjacent to $v$, let $\mf{u}$ be a hyperplane closest to $v$ among those that separate $v$ from $\mf{w}$; otherwise, let us set $\mf{u}=\mf{w}$. If there existed a hyperplane $\mf{u}'$ separating $v$ and $\mf{u}$, we would have $d(v,\mf{u}')<d(v,\mf{u})$ and $\mf{u}'$ would separate $v$ and $\mf{w}$; this would contradict our choice of $\mf{u}$. We conclude that $\mf{u}$ is adjacent to $v$ and, since $\mf{u}$ must lie in both $\mscr{W}(r_i)$, we have $\mf{u}\in\mscr{W}_1\cap\mscr{W}_2$.
\end{proof}

A subset $\s\cu\mscr{H}$ is an \emph{ultrafilter} if it satisfies the following two conditions:
\begin{enumerate}
\item given any two halfspaces $\mf{h},\mf{k}\in\s$, we have $\mf{h}\cap\mf{k}\neq\emptyset$;
\item for any hyperplane $\mf{w}\in\mscr{W}$, a side of $\mf{w}$ lies in $\s$.
\end{enumerate}
We say that $\s$ is a \emph{DCC ultrafilter} if, moreover, every descending chain of halfspaces in $\s$ is finite. For every vertex $v\in X$, we denote by $\s_v\cu\mscr{H}$ the set of halfspaces containing $v$. This is a DCC ultrafilter.

Consider now the map $\iota\colon X\ra 2^{\mscr{H}}$ taking each vertex $v$ to the set $\s_v$. The image $\iota (X)$ coincides with the collection of all DCC ultrafilters. Endowing $2^{\mscr{H}}$ with the product topology, we can consider the closure $\overline{\iota(X)}$, which is precisely the set of all ultrafilters. Equipped with the subspace topology, this is a compact Hausdorff space known as the \emph{Roller compactification} of $X$; we denote it by $\overline X$. The \emph{Roller boundary} is $\partial X:=\overline X\setminus X$. 

We prefer to imagine $\partial X$ as a set of points at infinity, rather than a set of ultrafilters. We will therefore write $x\in\partial X$ for points in the Roller boundary and employ the notation $\s_x\cu\mscr{H}$ to refer to the ultrafilter representing $x$.

Although this will not be needed in the present paper, it is interesting to observe that the Roller boundary $\partial X$ is naturally homeomorphic to the horofunction boundary of the metric space $(X,d)$. This is an unpublished result of U.\ Bader and D.\ Guralnik; see \cite{Caprace-Lecureux} or \cite{Fernos-Lecureux-Matheus} for a proof. Note that, on the other hand, the horofunction boundary with respect to the $\CAT$ metric would simply coincide with the visual boundary of $X$.

Given a (combinatorial) ray $r\cu X$ and a hyperplane $\mf{w}\in\mscr{W}$, there exists a unique side $\mf{h}$ of $\mf{w}$ such that $r\setminus\mf{h}$ is bounded. The collection of all such halfspaces forms an ultrafilter and we denote by $r^+\in\partial X$ the corresponding point; we refer to $r^+$ as the \emph{endpoint at infinity} of $r$. 

Fixing a basepoint $v\in X$, every point of $\partial X$ is of the form $r^+$ for a ray $r$ based at $v$. This yields a bijection between points of $\partial X$ and rays based at $v$, where we need to identify the rays $r_1$ and $r_2$ whenever $\mscr{W}(r_1)=\mscr{W}(r_2)$. See Proposition~A.2 in \cite{Genevois} for details.

Note that, given $v\in X$ and $\mf{h}\in\mscr{H}$, we have $v\in\mf{h}$ if and only if $\mf{h}\in\s_v$. We thus extend the halfspace $\mf{h}\cu X$ to a subset $\overline{\mf{h}}\cu\overline X$ by declaring that a point $x\in\partial X$ lies in $\overline{\mf{h}}$ if and only if $\mf{h}\in\s_x$. In particular, $\overline{\mf{h}}$ and $\overline{\mf{h}^*}$ provide a partition of $\overline X$ with $\overline{\mf{h}}\cap X=\mf{h}$ and $\overline{\mf{h}^*}\cap X=\mf{h}^*$. For ease of notation, we will generally omit the overline symbol and will not distinguish between a halfspace $\mf{h}\cu X$ and its extension $\overline{\mf{h}}\cu\overline X$.

Given subsets $A,B\cu\overline X$, we employ the notation:
\begin{align*}
&\mscr{H}(A|B)=\{\mf{h}\in\mscr{H}\mid B\cu\mf{h},~A\cu\mf{h}^*\}, \\
&\mscr{W}(A|B)=\{\mf{w}\in\mscr{W}\mid \text{a side of $\mf{w}$ lies in $\mscr{H}(A|B)$}\}.
\end{align*}
It is immediate from the definitions that $\mscr{W}(x|y)\neq\emptyset$ if and only if $x,y\in\overline X$ are distinct. If $u,v\in X$ are vertices, we have $d(u,v)=\#\mscr{W}(u|v)$.  

Given $x,y\in\overline X$, the \emph{interval} between $x$ and $y$ is the set
\[I(x,y)=\{z\in\overline X\mid\mscr{W}(z|x,y)=\emptyset\}.\]
We always have $I(x,x)=\{x\}$. In general, $I(x,y)\cap X$ coincides with the union of all (possibly infinite) geodesics with endpoints $x$ and $y$. In particular, if $u,v,w\in X$ are vertices, we have $w\in I(u,v)$ if and only if $d(u,v)=d(u,w)+d(w,v)$. 

For any three points $x,y,z\in\overline X$, there exists a unique $m(x,y,z)\in\overline X$ that lies in all three intervals $I(x,y)$, $I(y,z)$ and $I(z,x)$. We refer to $m(x,y,z)$ as the \emph{median} of $x$, $y$ and $z$ and remark that it is represented by the ultrafilter
\[(\s_x\cap\s_y)\cup(\s_y\cap\s_z)\cup(\s_z\cap\s_x).\] 
If $v_1,v_2,v_3\in X$, the median $m=m(v_1,v_2,v_3)$ is the only vertex satisfying $d(v_i,v_j)=d(v_i,m)+d(m,v_j)$ for all $1\leq i<j\leq 3$. The operator $m$ determines a continuous map $m\colon\overline{X}^3\ra\overline X$ that endows $\overline X$ with a structure of \emph{median algebra}. See e.g.\ \cite{Roller} for a definition of the latter notion.

Given $x,y,z\in\overline X$, the median $m=m(x,y,z)$ is the only point of $I(x,y)$ with the property that $m\in I(z,w)$ for every $w\in I(x,y)$. In particular, $m$ is the unique point of $I(x,y)$ that is closest to $z$. For this reason, we also refer to $m(x,y,z)$ as the \emph{gate-projection} of $z$ to $I(x,y)$.

\subsection{Extremal vertices and straight geodesics.}

Let $X$ be a $\CAT$ cube complex. We introduce the following two notions.

\begin{defn}\label{sharp corners defn}
A vertex $v\in X$ is \emph{extremal} if there exists an edge $e\cu X$ incident to $v$ such that any other edge incident to $v$ spans a square with $e$.
\end{defn}

Equivalently, $\lk(v)$ is a cone over one of its vertices. This happens if and only if a neighbourhood of $v$ splits as $[0,1)\x N$ for a subcomplex $N\cu X$.

\begin{defn}\label{straight geodesic defn}
A geodesic $\g\cu X$ is \emph{straight} if no two hyperplanes in $\mscr{W}(\g)$ are transverse.
\end{defn}

Our interest in cube complexes with no extremal vertices is motivated by the following straightforward observation (proof omitted).

\begin{lem}\label{straight rays} 
If $X$ has no extremal vertices, every edge can be extended to a straight (bi-infinite) line. 
\end{lem}

We say that $X$ is \emph{complete} if there is no infinite ascending chain of cubes (cf.\ \cite{Leary}). In particular, finite dimensional cube complexes are always complete. A \emph{free face} in $X$ is a non-maximal cube $c\cu X$ that is contained in a unique maximal cube.

\begin{rmk}\label{no extremal vertices}
If $X$ is complete and it has no free faces, then $X$ has no extremal vertices. Indeed, consider a vertex $v\in X$ contained in an edge $e$. Since $X$ is complete, there exists a maximal cube $c_1\cu X$ containing $e$. Let $c\cu c_1$ be the face such that $v\in c$ and $c_1\simeq c\x e$. Since $c$ is not a free face, there exists a maximal cube $c_2\cu X$ with $c_1\cap c_2=c$. Let $e'\cu c_2$ be an edge such that $v\in e'$ and $e'\not\cu c$. The edges $e$ and $e'$ do not span a square or $e'$ and $c_1$ would span a cube properly containing $c_1$. Hence $v$ is not an extremal vertex.
\end{rmk}

Nevertheless, the reader will realise that $\CAT$ cube complexes with no extremal vertices are much more common than cube complexes with no free faces. For instance, the (universal cover of) the Davis complex \cite{Davis} associated to a right-angled Coxeter group $G$ often has free faces\footnote{More precisely, this happens if and only if the defining flag complex has free faces.}, but it only has extremal vertices when $G\simeq\Z/2\Z\x H$ for a parabolic subgroup $H$.

\section{Cross ratios on cube complexes.}\label{cross ratio section}

Let $X$ be a $\CAT$ cube complex with \emph{combinatorial metric} $d$. Given a base vertex $v\in X$, the \emph{Gromov product} of $x,y\in\overline X$ is given by:
\[(x\cdot y)_v:=\#\mscr{W}(v|x,y)=d\big(v,m(v,x,y)\big)\in\N\cup\{+\infty\}.\]
Note that $(x\cdot y)_v=+\infty$ if and only if $m(v,x,y)\in\partial X$. Whenever $x,y\in X$, the above quantity coincides with the usual Gromov product:
\[(x\cdot y)_v=\tfrac{1}{2}\cdot\big[d(v,x)+d(v,y)-d(x,y)\big].\] 
The following simple observation can be found as Lemma~2.3 in \cite{BF1}.

\begin{lem}\label{infinite Gromov product}
Consider $x,y,z\in\overline X$ and $v\in X$.
\begin{enumerate}
\item We have $m(x,y,z)\in X$ if and only if each of the three intervals $I(x,y)$, $I(y,z)$, $I(z,x)$ intersects $X$.
\item We have $(x\cdot y)_v<+\infty$ if and only if $I(x,y)$ intersects $X$.
\end{enumerate}
\end{lem}

Fixing $v\in X$, we consider the subset $\mscr{A}\cu(\overline X)^4$ of $4$--tuples $(x,y,z,w)$ such that at most one of the three values $(x\cdot y)_v+(z\cdot w)_v$, $(x\cdot z)_v+(y\cdot w)_v$ and $(x\cdot w)_v+(y\cdot z)_v$ is infinite. By part~$(2)$ of Lemma~\ref{infinite Gromov product}, the set $\mscr{A}$ does not depend on the choice of $v$. The map $\Cr_v\colon\mscr{A}\ra\Z\cup\{\pm\infty\}$ defined by: 
\[\Cr_v(x,y,z,w)=(x\cdot z)_v+(y\cdot w)_v-(x\cdot w)_v-(y\cdot z)_v\]
satisfies the following identities for all $4$--tuples $(x,y,z,w)$, $(x,y,z,t)$ and $(x,y,t,w)$ in $\mscr{A}$:
\begin{enumerate}
\item[(i)] $\Cr_v(x,y,z,w)=-\Cr_v(y,x,z,w)$;
\item[(ii)] $\Cr_v(x,y,z,w)=\Cr_v(z,w,x,y)$;
\item[(iii)] $\Cr_v(x,y,z,w)=\Cr_v(x,y,z,t)+\Cr_v(x,y,t,w)$;
\item[(iv)] $\Cr_v(x,y,z,w)+\Cr_v(y,z,x,w)+\Cr_v(z,x,y,w)=0$.
\end{enumerate}
The next result shows that $\Cr_v$ is moreover basepoint-independent.

\begin{prop}\label{independent of basepoint}
For every $v\in X$ and every $(x,y,z,w)\in\mscr{A}$, we have
\[\Cr_v(x,y,z,w)=\#\mscr{W}(x,z|y,w)-\#\mscr{W}(x,w|y,z).\]
\end{prop}
\begin{proof}
We show that every hyperplane $\mf{w}\in\mscr{W}$ gives the same contribution to both sides of the equality. Note that
\[\Cr_v(x,y,z,w)=\#\mscr{W}(v|x,z)+\#\mscr{W}(v|y,w)-\#\mscr{W}(v|x,w)-\#\mscr{W}(v|y,z).\]
Every $\mf{w}\in\mscr{W}(x,z|y,w)$ contributes to either $\mscr{W}(v|x,z)$ or $\mscr{W}(v|y,w)$ by $+1$, without affecting $\mscr{W}(v|x,w)$ and $\mscr{W}(v|y,z)$. Similarly, every hyperplane ${\mf{w}\in\mscr{W}(x,w|y,z)}$ decreases $-\mscr{W}(v|x,w)-\mscr{W}(v|y,z)$ by $1$ and leaves $\mscr{W}(v|x,z)$ and $\mscr{W}(v|y,w)$ invariant. Thus, it suffices to check that hyperplanes ${\mf{w}\not\in\mscr{W}(x,z|y,w)\sqcup\mscr{W}(x,w|y,z)}$ do not affect $\Cr_v(x,y,z,w)$. 

This is clear if all four points $x$, $y$, $z$ and $w$ lie on the same side of $\mf{w}$, or if $\mf{w}\in\mscr{W}(x,y|z,w)$. The remaining case is when exactly three of the four points lie on one side of $\mf{w}$. Performing a sequence of moves $(x\leftrightarrow y,z\leftrightarrow w)$ and $(x\leftrightarrow z,y\leftrightarrow w)$, which leave $\Cr_v(x,y,z,w)$ invariant, we reduce to the case when $\mf{w}\in\mscr{W}(x|y,z,w)$. If $v$ is not on the same side of $\mf{w}$ as $x$, the hyperplane $\mf{w}$ does not contribute to any summand of $\Cr_v(x,y,z,w)$. Otherwise ${\mf{w}\in\mscr{W}(x,v|y,z,w)}$; in this case the only contributions to $\Cr_v(x,y,z,w)$ arise from $\mscr{W}(v|y,w)$ and $\mscr{W}(v|y,z)$ and they cancel each other.
\end{proof}

We remark that the right-hand side of the equality in Proposition~\ref{independent of basepoint} is in general defined on a set strictly larger than $\mscr{A}$.

\begin{cor}\label{independent of basepoint 2}
The map $\Cr_v\colon\mscr{A}\ra\Z\cup\{\pm\infty\}$ is independent of the choice of $v$. All automorphisms of $X$ preserve $\Cr_v$.
\end{cor}

\begin{defn}
We will write $\Cr\colon\mscr{A}\ra\Z\cup\{\pm\infty\}$ from now on and refer to it as \emph{the cross ratio} on $\overline X$ (or $\partial X$).
\end{defn}

Identities~(i) and~(ii) imply that $|\Cr(x,y,z,w)|$ is invariant under a subgroup of order $8$ of $\Sym(\{x,y,z,w\})$. Thus, we only need to record $24/8=3$ `meaningful' values for every subset $\{x,y,z,w\}$. These values are precisely the three cross ratios appearing in identity~(iv), so they are not independent. 

The purpose of Definition~\ref{crt defn} below is precisely to record simultaneously all cross ratios obtained by permuting coordinates. We first introduce some notation. Given $a_1,b_1,c_1,a_2,b_2,c_2\in\N\cup\{+\infty\}$, we declare the triples $(a_1,b_1,c_1)$ and $(a_2,b_2,c_2)$ to be \emph{equivalent} if there exists $n\geq 0$ such that 
\[a_i=a_j+n;~~~~b_i=b_j+n;~~~~c_i=c_j+n,\] 
where $\{i,j\}=\{1,2\}$. The equivalence class of the triple $(a,b,c)$ is denoted by $\ll a:b:c\rr$. Note that $\ll+\infty:+\infty:+\infty\rr$ is the only class consisting of a single triple; every other equivalence class has a unique representative with at least one zero entry. We also remark that all triples in a given equivalence class have the same infinite entries.

\begin{defn}\label{crt defn}
Given $x,y,z,w\in\overline X$ and $v\in X$, the \emph{cross ratio triple} $\crt_v(x,y,z,w)$ is the equivalence class
\[\Big\ll(x\cdot y)_v+(z\cdot w)_v:(x\cdot z)_v+(y\cdot w)_v:(x\cdot w)_v+(y\cdot z)_v\Big\rr.\]
\end{defn}

Note that $\crt_v$ is always independent of the choice of $v$. This follows from Corollary~\ref{independent of basepoint 2} when $(x,y,z,w)\in\mscr{A}$ and is clear otherwise. We are therefore allowed to simply write $\crt$.

All entries of a cross ratio triple are nonnegative. The three cross ratios in identity~(iv) above are recovered by taking the difference of two entries of the triple. We will employ asterisks $\ast$ when we do not want to specify a coordinate of $\crt(x,y,z,w)$. For instance, we write $\crt(x,y,z,w)=\ll\ast:0:1\rr$ rather than $\crt(x,y,z,w)=\ll a:0:1\rr$ and $a\in\N\cup\{+\infty\}$.

Let now $Y$ be another $\CAT$ cube complex. We write $\mscr{A}(X)$, rather than just $\mscr{A}$, when it is necessary to specify the cube complex under consideration. The following makes the notion of \emph{``cross-ratio preserving''} map more precise. 

\begin{defn}
A map $f\colon\partial X\ra\partial Y$ is \emph{M\"obius} if, for all $x,y,z,w\in\partial X$ with $(x,y,z,w)\in\mscr{A}(X)$, we have $(f(x),f(y),f(z),f(w))\in\mscr{A}(Y)$ and 
\[\Cr(f(x),f(y),f(z),f(w))=\Cr(x,y,z,w).\] 
The latter happens if and only if $\crt(f(x),f(y),f(z),f(w))=\crt(x,y,z,w)$ for all $(x,y,z,w)\in\mscr{A}(X)$.
\end{defn}

We remark that a bijection $f\colon\partial X\ra\partial Y$ is M\"obius if and only if its inverse $f^{-1}\colon\partial Y\ra\partial X$ is.

\section{M\"obius bijections between Roller boundaries.}

This section is devoted to the proof of the Main Theorem. Throughout it, let $X$ and $Y$ be $\CAT$ cube complexes with no extremal vertices. We moreover consider a M\"obius bijection $f\colon\partial X\ra\partial Y$.

To avoid cumbersome formulas, we will employ the following notation for $x,y,z,w\in\partial X$ and $v\in Y$:
\[(x\cdot y)_v^f=(f(x)\cdot f(y))_v, \hspace{.5cm} m^f(x,y,z):=m(f(x),f(y),f(z)),\]
\[\crt^f(x,y,z,w)=\crt(f(x),f(y),f(z),f(w)).\]

\subsection{Opposite points.}\label{dist pres sect}

As described in the introduction, the following notion will be crucial to avoid the issues depicted in Figure~\ref{spanning a cube}.

\begin{defn}[Definition~5.2 in \cite{BF1}]\label{opposite defn}
Given $x,y,z\in\overline X$, we say that $x$ and $y$ are \emph{opposite} with respect to $z$ (written $x\op_z y$) if the median ${m=m(x,y,z)}$ lies in $X$ and $I(x,y)=I(x,m)\cup I(m,y)$.
\end{defn}

We will also write $x\op_z^fy$ with the meaning of $f(x)\op_{f(z)}f(y)$.

\begin{rmk}
Let $x,y,z\in\overline X$ be points with $m=m(x,y,z)\in X$; denote by $\mscr{W}_m\cu\mscr{W}(X)$ the subset of hyperplanes adjacent to $m$. Lemma~5.1 in \cite{BF1} shows that we have $x\op_zy$ if and only if no element of $\mscr{W}(m|x)\cap\mscr{W}_m$ is transverse to an element of $\mscr{W}(m|y)\cap\mscr{W}_m$.
\end{rmk}

\begin{rmk}\label{one is zero}
Consider points $x,y,z\in\overline X$ with $x\op_z y$ and $m=m(x,y,z)$. Given any $w\in\partial X$, we either have $(x\cdot w)_m=0$ or $(y\cdot w)_m=0$. Indeed, the gate-projection $m(x,y,w)$ falls either in $I(y,m)$ or in $I(x,m)$.
\end{rmk}

Our goal is now to show that the property in Definition~\ref{opposite defn} is preserved by the M\"obius bijection $f$. We will rely on the following analogue of Proposition~5.4 in \cite{BF1}.

\begin{lem}\label{op vs lcrt}
Given points $x_1,x_2,y\in\partial X$ with $m(x_1,x_2,y)\in X$, the condition $x_1\op_y x_2$ fails if and only if there exists a point $z\in\partial X$ such that ${\crt(x_1,x_2,y,z)=\ll a:b:c\rr}$ and ${a<\min\{b,c\}<+\infty}$.
\end{lem}
\begin{proof}
Setting $m=m(x_1,x_2,y)$, we have:
\[\crt_m(x_1,x_2,y,z)=\big\ll (y\cdot z)_m:(x_2\cdot z)_m:(x_1\cdot z)_m\big\rr.\] 
If $x_1\op_y x_2$ and $z\in\partial X$, Remark~\ref{one is zero} yields $\min\{(x_2 \cdot z)_m,(x_1 \cdot z)_m\}=0$ and we cannot have $(y\cdot z)_m<0$.

Consider instead the case when $x_1$ and $x_2$ are not opposite with respect to $y$. There exist transverse hyperplanes $\mf{w}_i\in\mscr{W}(m|x_i)$ adjacent to $m$. Denote by $m'\in X$ the vertex with ${\mscr{W}(m|m')=\{\mf{w}_1,\mf{w}_2\}}$. By Lemma~\ref{straight rays}, there exists a straight ray $r$ such that $r(1)=m'$ and $\mscr{W}(r(0)|r(1))=\{\mf{w}_1\}$; we set $z=r^+$. Let $\g$ and $\g_2$ be rays based at $r(0)$ satisfying $\g(1)=m$, $\g^+=y$ and $\g_2^+=x_2$. As $r(0)$, $y$ and $x_2$ are all on the same side of $\mf{w}_1$, Lemma~\ref{straight rays 2} implies that the unions $\g\cup r$ and $\g_2\cup r$ are lines. Hence $(y\cdot z)_m=0$ and $(x_2\cdot z)_m=1$. Since ${\mf{w}_1\in\mscr{W}(m|x_1,z)}$, we also have $(x_1\cdot z)_m\geq 1$. We conclude that $\crt(x_1,x_2,y,z)=\ll 0:1:c\rr$ with $c\geq 1$.
\end{proof}

\begin{prop}\label{invariance of op}
Given $x,y,z\in\partial X$, we have $x\op_z y$ if and only if $x\op_z^f y$.
\end{prop}
\begin{proof}
Fix a basepoint $v\in X$. Assume that $x\op_zy$; in particular, we have $m(x,y,z)\in X$. By Lemma~\ref{infinite Gromov product}, the latter is equivalent to the Gromov products $(x\cdot y)_v$, $(y\cdot z)_v$ and $(z\cdot x)_v$ being all finite. In other words, the $4$--tuples $(x,x,y,y)$, $(y,y,z,z)$ and $(z,z,x,x)$ all lie in $\mscr{A}(X)$. As $f$ is M\"obius, it takes these $4$--tuples into $\mscr{A}(Y)$ and we must have $m^f(x,y,z)\in Y$.

Now, if we did not have $x\op_z^fy$, Lemma~\ref{op vs lcrt} would yield $w\in\partial Y$ with $\crt(f(x),f(y),f(z),w)=\ll a:b:c\rr$ and $a<\min\{b,c\}<+\infty$. In particular $(f(x),f(y),f(z),w)\in\mscr{A}(X)$ and hence $\crt(x,y,z,f^{-1}(w))=\ll a:b:c\rr$, contradicting Lemma~\ref{op vs lcrt}. Thus $x\op_z y\Ra x\op_z^f y$ and the converse implication follows by considering $f^{-1}\colon B\ra A$.
\end{proof}

We can use triples of opposite points to obtain a well-defined map $X\ra Y$. We now describe this procedure, culminating in Corollary~\ref{distance is preserved} below.

The next three results also appear in \cite{BF1} as Lemmas~5.21,~5.22 and Proposition~5.23. We include them here along with their proofs for the convenience of the reader, but also because the standing assumptions of \cite{BF1} are much stronger than the current ones. 

Given points $x_1,x_2,x,y_1,y_2,y\in\partial X$ with $x_1\op_x x_2$ and $y_1\op_y y_2$, we set:
\begin{align*} 
&m_x=m(x_1,x_2,x), & &m_y=m(y_1,y_2,y), \\ 
&m_x'=m^f(x_1,x_2,x), & &m_y'=m^f(y_1,y_2,y).
\end{align*}

\begin{lem}\label{order is preserved}
Given $u\in\partial X$ with $(x_1,x_2,x,u)\in\mscr{A}(X)$, we have:
\[(x_1\cdot u)_{m_x}=(x_1\cdot u)^f_{m_x'}, \hspace{.5cm} (x_2\cdot u)_{m_x}=(x_2\cdot u)^f_{m_x'}, \hspace{.5cm} (x\cdot u)_{m_x}=(x\cdot u)^f_{m_x'}.\]
\end{lem}
\begin{proof}
Observe that:
\[\crt_{m_x}(x_1,x_2,x,u)=\ll (x\cdot u)_{m_x}:(x_2\cdot u)_{m_x}:(x_1\cdot u)_{m_x}\rr,\]
\[\crt^f_{m_x'}(x_1,x_2,x,u)=\ll (x\cdot u)^f_{m_x'}:(x_2\cdot u)^f_{m_x'}:(x_1\cdot u)^f_{m_x'}\rr.\]
Since $x_1\op_x x_2$, Remark~\ref{one is zero} shows that either $(x_1\cdot u)_{m_x}=0$ or $(x_2\cdot u)_{m_x}=0$. Since $x_1\op_x^f x_2$ by Proposition~\ref{invariance of op}, also one among $(x_1\cdot u)^f_{m_x'}$ and $(x_2\cdot u)^f_{m_x'}$ must vanish. The equality $\crt(x_1,x_2,x,u)=\crt^f(x_1,x_2,x,u)$ then implies that $(x_1\cdot u)_{m_x}=(x_1\cdot u)^f_{m_x'}$, $(x_2\cdot u)_{m_x}=(x_2\cdot u)^f_{m_x'}$ and $(x\cdot u)_{m_x}=(x\cdot u)^f_{m_x'}$.
\end{proof}

\begin{lem}\label{more preserved products}
Let $u,v\in\partial X$ be two points such that the $4$--tuples $(x_1,x_2,u,v)$, $(x_1,x_2,x,u)$ and $(x_1,x_2,x,v)$ all lie in $\mscr{A}(X)$. Then $(u\cdot v)_{m_x}=(u\cdot v)^f_{m_x'}$.
\end{lem}
\begin{proof}
We have $\crt(x_1,x_2,u,v)=\crt^f(x_1,x_2,u,v)$. Equating the cross ratio triples $\crt_{m_x}(x_1,x_2,u,v)$ and $\crt_{m_x'}^f(x_1,x_2,u,v)$, we obtain
\[\ll(u\cdot v)_{m_x}: b:c\rr=\ll (u\cdot v)^f_{m_x'}: b':c'\rr,\]
where Lemma~\ref{order is preserved} yields $b=b'$ and $c=c'$. Hence $(u\cdot v)_{m_x}=(u\cdot v)^f_{m_x'}$.
\end{proof}

\begin{prop}\label{d(mx,my)}
We have:
\[d(m_x,m_y)=(y_1\cdot y_2)_{m_x}+\left|(y_1\cdot y)_{m_x}-(y_2\cdot y)_{m_x}\right|.\]
\end{prop}
\begin{proof}
Set $v=m(y_1,y_2,m_x)$. As $v$ is the gate-projection of $m_x$ to the interval $I(y_1,y_2)$, we have:
\[d(m_x,m_y)=d(m_x,v)+d(v,m_y),\]
where $d(m_x,v)=(y_1\cdot y_2)_{m_x}$. Up to exchanging $y_1$ and $y_2$, we can assume that $v$ lies within $I(m_y,y_2)$. Since no element of $\mscr{W}(v|y_2)=\mscr{W}(m_x,y_1|y_2)$ separates $m_x$ and $y$, it follows that the set $\mscr{W}(m_x,y_1|y_2,y)$ is empty. We conclude that $(y_2\cdot y)_{m_x}=\#\mscr{W}(m_x|y_1,y_2,y)$. On the other hand, observing that $\mscr{W}(v|m_y)=\mscr{W}(m_x,y_2|y_1,y)$, we have 
\[\mscr{W}(m_x|y_1,y)=\mscr{W}(m_x|y_1,y_2,y)\sqcup\mscr{W}(v|m_y)\]
and $(y_1\cdot y)_{m_x}=(y_2\cdot y)_{m_x}+d(v,m_y)$. 
\end{proof}

Lemma~\ref{more preserved products} and Proposition~\ref{d(mx,my)} immediately yield the following.

\begin{cor}\label{distance is preserved} 
Suppose that $(x_1,x_2,u,v)$ and $(x_1,x_2,x,u)$ lie in $\mscr{A}(X)$ whenever $u$ and $v$ are distinct elements of the set $\{y_1,y_2,y\}$. Then, we have $d(m_x,m_y)=d(m_x',m_y')$. In particular, $m_x'$ and $m_y'$ coincide if and only if $m_x$ and $m_y$ do.
\end{cor}

\subsection{Straight points.}

In order to ensure that the hypotheses of Corollary~\ref{distance is preserved} are satisfied, we will consider the following class of boundary points.

\begin{defn}
A point $x\in\partial X$ is \emph{straight} if there exists a straight ray $r\cu X$ with $r^+=x$; equivalently, $x$ is an endpoint of a straight line. We denote by $\partial_sX\cu\partial X$ the set of straight boundary points.
\end{defn}

Observe that two points $x,y\in\partial X$ are endpoints of a straight line $\g$ if and only if the interval $I(x,y)\cap X$ is isomorphic to $\R$. Indeed, $I(x,y)\cap X$ coincides with $\g$ in this case. The following result characterises such situations in a similar way to Lemma~\ref{op vs lcrt}.

\begin{lem}\label{invariance of straight geodesics}
Two points $x,y\in\partial X$ are endpoints of a straight line if and only if both the following are verified:
\begin{enumerate}
\item $I(x,y)\cap X\neq\emptyset$;
\item there do not exist points $z,w\in\partial X$ with $\crt(x,y,z,w)=\ll a:b:c\rr$ and $a<\min\{b,c\}<+\infty$.
\end{enumerate}
\end{lem}
\begin{proof}
We begin by assuming that $x$ and $y$ are endpoints of a straight line $\g$. Condition~$(1)$ is clearly satisfied and we are now going to prove condition~$(2)$ for points $z,w\in\partial X$. 

If $m(x,y,z)\in X$, then $x\op_zy$ and Lemma~\ref{op vs lcrt} shows that we cannot have $\crt(x,y,z,w)=\ll a:b:c\rr$ and $a<\min\{b,c\}<+\infty$. The same happens if $m(x,y,w)\in X$, as can be observed by simply swapping $z$ and $w$. We are left to examine the situation where $m(x,y,z)$ and $m(x,y,w)$ lie in the boundary; note that both medians must then belong to the set $\{x,y\}$. We can assume that $m(x,y,z)\neq m(x,y,w)$ as otherwise the first coordinate of $\crt(x,y,z,w)$ is infinite. If $m(x,y,z)=x$ and $m(x,y,w)=y$, we have $\crt(x,y,z,w)=\ll\ast:\infty:0\rr$; otherwise, $\crt(x,y,z,w)=\ll\ast:0:\infty\rr$. In all cases $\crt(x,y,z,w)$ is not of the form $\ll a:b:c\rr$ with $a<\min\{b,c\}<+\infty$ and condition~$(2)$ is satisfied.

We now assume that $I(x,y)\cap X\neq\emptyset$, but $x$ and $y$ are not endpoints of a straight line. We will show that condition~$(2)$ fails. The intersection $I(x,y)\cap X$ cannot be one-dimensional or it would be isomorphic to $\R$. Hence $I(x,y)\cap X$ contains a square $s$; denote by $\mf{w}$ and $\mf{w}'$ its hyperplanes. Let $v_x$ and $v_y$ be the vertices of $s$ such that $\{\mf{w},\mf{w}'\}$ is disjoint from $\mscr{W}(x|v_x)$ and $\mscr{W}(y|v_y)$. Lemma~\ref{straight rays} shows that there exist straight rays $r_x$ and $r_y$, based at $v_x$ and $v_y$ respectively, such that their first crossed hyperplane is $\mf{w}$. We set $z=r_x^+$ and $w=r_y^+$.

Lemma~\ref{straight rays 2} implies that $(z\cdot w)_{v_x}=0$, $(x\cdot z)_{v_x}=0$ and $(y\cdot w)_{v_x}=1$. Moreover, $(y\cdot z)_{v_x}\geq 1$ as $\mf{w}$ separates $v_x$ from $y$ and $z$. We conclude that $\crt(x,y,z,w)=\ll 0:1:c\rr$ with $c=(x\cdot w)_{v_x}+(y\cdot z)_{v_x}\geq 1$.
\end{proof}

\begin{prop}\label{invariance of straight}
\begin{enumerate}
\item We have $x\in\partial_sX$ if and only if $f(x)\in\partial_sY$.
\item If $x,y\in\partial X$ are endpoints of a straight line, so are the points $f(x)$ and $f(y)$.
\end{enumerate}
\end{prop}
\begin{proof}
As a boundary point is straight if and only if it is an endpoint of a straight line, part~$(1)$ follows from part~$(2)$. If $x,y\in\partial_sX$ are endpoints of a straight line $\g\cu X$, we have $I(x,y)\cap X\neq\emptyset$. By part~$(2)$ of Lemma~\ref{infinite Gromov product}, this is equivalent to the fact that $(x,x,y,y)$ lies in $\mscr{A}(X)$. We conclude that $I(f(x),f(y))\cap Y\neq\emptyset$.

Now, if $f(x)$ and $f(y)$ were not endpoints of a straight line, Lemma~\ref{invariance of straight geodesics} would yield points $z,w\in\partial Y$ with $\crt(f(x),f(y),z,w)=\ll a:b:c\rr$ and ${a<\min\{b,c\}<+\infty}$. However, $\crt(x,y,f^{-1}(z),f^{-1}(w))$ would then have the same form, contradicting Lemma~\ref{invariance of straight geodesics}.
\end{proof}

The next result is our main motivation for considering straight points.

\begin{lem}\label{Gromov products of straight points}
Consider $x\in\partial_sX$ and a vertex $v\in X$. Given $y,z\in\partial_sX$ with $(y\cdot z)_v<+\infty$, at least one of the Gromov products $(x\cdot y)_v$, $(x\cdot z)_v$ is finite.
\end{lem}
\begin{proof}
Let $r_x$, $r_y$ and $r_z$ be straight rays representing $x$, $y$ and $z$, respectively. As the symmetric difference $\mscr{W}(r_x)\triangle\mscr{W}(v|x)$ is contained in $\mscr{W}(r_x(0)|v)$, the intersection $\mc{U}_x=\mscr{W}(r_x)\cap\mscr{W}(v|x)$ is cofinite in $\mscr{W}(v|x)$ and does not contain transverse hyperplanes. The same holds for $\mc{U}_y=\mscr{W}(r_y)\cap\mscr{W}(v|y)$ and $\mc{U}_z=\mscr{W}(r_z)\cap\mscr{W}(v|z)$. 

If we had $(x\cdot y)_v=(x\cdot z)_v=+\infty$, the set $\mscr{W}(v|x)$ would have infinite intersection with both $\mscr{W}(v|y)$ and $\mscr{W}(v|z)$. In particular, both $\mc{U}_x\cap\mc{U}_y$ and $\mc{U}_x\cap\mc{U}_z$ would be infinite. As any hyperplane separating two elements of $\mscr{W}(r_y)$ must lie in $\mscr{W}(r_y)$, the intersections $\mc{U}_x\cap\mc{U}_y$ and $\mc{U}_x\cap\mc{U}_z$ would then be cofinite in $\mc{U}_x$. Hence $\mc{U}_y\cap\mc{U}_z$ would be infinite, contradicting the fact that $(y\cdot z)_v<+\infty$.
\end{proof}

\begin{cor}\label{isom for straight}
Given points $x_1,x_2,x,y_1,y_2,y\in\partial_sX$ with $x_1\op_xx_2$ and $y_1\op_yy_2$, we have:
\[d(m(x_1,x_2,x),m(y_1,y_2,y))=d(m^f(x_1,x_2,x),m^f(y_1,y_2,y)).\]
\end{cor}
\begin{proof}
We only need to verify the hypotheses of Corollary~\ref{distance is preserved}. To this end, let $u$ and $v$ be distinct elements of the set $\{y_1,y_2,y\}$ and fix a basepoint $p\in X$. As the Gromov products $(x_1\cdot x_2)_p$, $(x_1\cdot x)_p$ and $(x_2\cdot x)_p$ are all finite, Lemma~\ref{Gromov products of straight points} shows that $(x_1,x_2,x,u)\in\mscr{A}(X)$. If $(x_1,x_2,u,v)$ does not lie in $\mscr{A}(X)$, we can assume, up to permuting the points, that either $(x_1\cdot u)_p=(x_2\cdot u)_p=+\infty$ or $(x_1\cdot u)_p=(x_1\cdot v)_p=+\infty$. As $(x_1\cdot x_2)_p$ and $(u\cdot v)_p$ are finite, both situations are ruled out by Lemma~\ref{Gromov products of straight points}. 
\end{proof}

\subsection{Skinny vertices.}

We now address the problem of which vertices $v\in X$ can be represented as median of a triple such as those in Corollary~\ref{isom for straight}. Throughout this subsection, $X$ is required to have at least two vertices. 

We say that a vertex $v\in X$ is \emph{skinny} if $\deg(v)=2$. Skinny vertices are exactly those that are `invisible from the boundary', as we now describe.

\begin{lem}\label{skinny points}
For a vertex $v\in X$, the following are equivalent:
\begin{enumerate}
\item $v$ is not skinny;
\item there exist $x,y,z\in\partial_sX$ such that $m(x,y,z)=v$ and $x\op_z y$.
\end{enumerate}
\end{lem}
\begin{proof}
If there exist $x,y,z\in\partial X$ with $m(x,y,z)=v$, the rays from $v$ to $x$, $y$ and $z$ must begin with three pairwise distinct edges. Hence $\deg(v)\geq 3$, which shows $(2)\Ra (1)$. 

Regarding the implication $(1)\Ra (2)$, we have $\deg(v)\geq 3$ as soon as $v$ is not skinny. Indeed, $\deg(v)=0$ can only happen if $X$ is a single point and $\deg(v)=1$ would violate the assumption that $X$ has no extremal vertices. Let $e_1$, $e_2$ and $e_3$ be pairwise distinct edges incident to $v$; as $v$ is not extremal, we can assume that $\mf{w}(e_1)$ and $\mf{w}(e_2)$ are not transverse. Lemma~\ref{straight rays} allows us to extend each $e_i$ to a straight ray $r_i$. By Lemma~\ref{straight rays 2}, each union $r_i\cup r_j$ is a line if $i\neq j$. Setting $x=r_1^+$, $y=r_2^+$ and $z=r_3^+$, we thus have $m(x,y,z)=v$. As $r_1\cup r_2$ is straight, we also have $x\op_zy$.
\end{proof}

Denote by $\mscr{V}\cu X$ the set of skinny vertices. Let $\mscr{S}\cu X$ be the union of all edges intersecting $\mscr{V}$. Let $\mscr{F}\cu X$ be the full subcomplex with vertex set $X\setminus\mscr{V}$. We call $\mscr{S}$ and $\mscr{F}$ the \emph{skinny} and \emph{fat parts} of $X$, respectively. We remark that $X=\mscr{F}\cup\mscr{S}$ and that every vertex in $\mscr{S}\setminus\mscr{F}$ is skinny. We will employ the notation $\mscr{V}(X)$, $\mscr{S}(X)$ and $\mscr{F}(X)$ when it is necessary to specify the cube complex under consideration.

If $X\not\simeq\R$, each connected component of $\mscr{S}$ is either a straight ray or a straight segment; we refer to these as \emph{skinny rays} and \emph{skinny segments}. Every skinny ray intersects $\mscr{F}$ at a single vertex; given $v\in\mscr{F}$, we denote by $\mscr{R}(v)\cu\partial_sX$ the set of endpoints at infinity of skinny rays based at $v$. 

\begin{lem}\label{skinny rays}
A vertex $v\in X$ and a point $x\in\partial X$ are endpoints of a skinny ray if and only if there exist $y,z\in\partial_sX$ with the following properties:
\begin{enumerate}
\item $m(x,y,z)=v$ and $x\op_zy$;
\item for every $w\in\partial_sX\setminus\{x\}$ we have $\crt(x,y,z,w)=\ll\ast:\ast:0\rr$.
\end{enumerate}
\end{lem}
\begin{proof}
Suppose that $v$ and $x$ are endpoints of a skinny ray $r$. Lemma~\ref{straight rays} allows us to extend $r$ to a straight line $\g$; let $y$ be the endpoint of $\g$ other than $x$. As $\deg(v)\geq 3$ by definition, we can construct a straight ray $r'$ based at $v$ and disjoint from $\g$. Setting $z=(r')^+$, we have $m(x,y,z)=v$ and $x\op_zy$. If $w\in\partial X$ and $w\neq x$, we must have $m(x,y,w)\not\in I(x,v)\setminus\{v\}$; Remark~\ref{one is zero} then shows that $m(x,y,w)\in I(v,y)$ and $(x\cdot w)_v=0$. As $(y\cdot z)_v$ also vanishes, $\crt(x,y,z,w)$ is of the form $\ll\ast:\ast:0\rr$ as required.

Conversely, suppose that $y$ and $z$ are given satisfying condition~$(1)$. If $I(x,v)\cap X$ is not a skinny ray, it contains a vertex $u\neq v$ with $\deg(u)\geq 3$. Let $\mf{w}_x\in\mscr{W}(u|x)$ and $\mf{w}_v\in\mscr{W}(u|v)$ be hyperplanes adjacent to $u$. Let $e$ be an edge incident to $u$ and not crossing $\mf{w}_x$ or $\mf{w}_v$; let $\g$ be a straight ray extending $e$ and set $w=\g^+$. We have $\mf{w}_x\in\mscr{W}(w|x)$ so $w\neq x$. Similarly, we have $\mf{w}_v\in\mscr{W}(w|v)$; this shows that $(x\cdot w)_v>0$ and, along with $x\op_vy$, it guarantees that $m(x,y,w)\in I(x,v)$ and $(y\cdot w)_v=0$. Thus condition~$(2)$ fails, as $\crt(x,y,z,w)=\ll\ast:0:c\rr$ with $c=(x\cdot w)_v>0$.
\end{proof}

\subsection{The isomorphism and its uniqueness.}

In this subsection, we complete the proof of the Main Theorem. We now require $X$ and $Y$ to be neither single points, nor isomorphic to $\R$. 

Consider a non-skinny vertex $v\in X$. Lemma~\ref{skinny points} provides three points ${x_1,x_2,x\in\partial_sX}$ with $m(x_1,x_2,x)=v$ and $x_1\op_xx_2$. We define a map $F\colon\mscr{F}(X)\ra\mscr{F}(Y)$ by setting $F(v)=m^f(x_1,x_2,x)$. Note that $F$ is well-defined and distance-preserving by Corollary~\ref{isom for straight}. 

Applying the same construction to the inverse $f^{-1}\colon\partial Y\ra\partial X$, we obtain $H\colon\mscr{F}(Y)\ra\mscr{F}(X)$. By Propositions~\ref{invariance of op} and~\ref{invariance of straight}, the compositions $F\o H$ and $H\o F$ are the identity. We conclude that $F$ is surjective and, in fact, an isometric bijection of fat parts.

\begin{thm}\label{skinny extension}
The map $F$ extends to a cubical isomorphism $F\colon X\ra Y$.
\end{thm}
\begin{proof}
Observe that two vertices $v_1,v_2\in\mscr{F}(X)$ are endpoints of a skinny segment if and only if there does not exist any $v_3\in\mscr{F}(X)\setminus\{v_1,v_2\}$ satisfying $d(v_1,v_2)=d(v_1,v_3)+d(v_3,v_2)$. Thus $v_1$ and $v_2$ are endpoints of a skinny segment of length $\ell$ if and only if $F(v_1)$ and $F(v_2)$ are. We can therefore isometrically extend $F$ over all skinny segments in $X$. 

We are left to deal with skinny rays. We conclude by showing that $f(\mscr{R}(v))=\mscr{R}(F(v))$ for all $v\in\mscr{F}(X)$. It suffices to prove the inclusion $f(\mscr{R}(v))\cu\mscr{R}(F(v))$ and then apply the same argument to $f^{-1}$. 

Consider $x\in\mscr{R}(v)$ and let $y,z\in\partial_sX$ be the points provided by Lem\-ma~\ref{skinny rays}. Since $x\in\partial_sX$, we have $m^f(x,y,z)=F(v)$ and, by Proposition~\ref{invariance of op}, also $x\op_z^fy$. Given $w\in\partial_sY\setminus\{f(x)\}$, the point $f^{-1}(w)$ lies in the set $\partial_sX\setminus\{x\}$ by Proposition~\ref{invariance of straight}. Lemma~\ref{skinny rays} shows that $\crt(x,y,z,f^{-1}(w))$ is of the form $\ll\ast:\ast:0\rr$ and, by Lemma~\ref{Gromov products of straight points}, the $4$--tuple $(x,y,z,f^{-1}(w))$ lies in $\mscr{A}(X)$. Hence $\crt(f(x),f(y),f(z),w)=\ll\ast:\ast:0\rr$ for all $w\in\partial_sY\setminus\{f(x)\}$. Lemma~\ref{skinny rays} finally implies that $f(x)\in\mscr{R}(F(v))$.
\end{proof}

Now, the isomorphism $F\colon X\ra Y$ extends to an isomorphism of median algebras $\overline F\colon\overline X\ra\overline Y$. We conclude the proof of the Main Theorem via:

\begin{thm}\label{main thm}
The map $F$ is the only cubical isomorphism with $\overline F|_{\partial X}=f$. 
\end{thm}
\begin{proof}
The uniqueness of $F$ is clear from our construction. We need to prove that $\overline F(x)=f(x)$ for every $x\in\partial X$. First, we suppose that $x\in\partial_sX$. 

Let $\g$ be a straight line with an endpoint at $x$; denote by $y$ the other endpoint of $\g$. We can assume that $x$ is not endpoint of a skinny ray, as $\overline F$ and $f$ clearly coincide on those. Thus, there exist vertices $v_n\in\g$ with $\deg(v_n)\geq 3$ and $v_n\ra x$; we can moreover assume that $v_{n+1}\in I(v_n,x)\setminus\{v_n\}$.

Let $e_n$ be an edge with $e_n\cap\g=\{v_n\}$. Extending $e_n$ to a straight ray, we construct $z_n\in\partial_sX$ with $m(x,y,z_n)=v_n$ and $x\op_{z_n}y$. Note that $\crt_{v_0}(x,y,z_0,z_n)=\ll\ast:0:c_n\rr$, where $c_n=(z_n\cdot x)_{v_0}=d(v_0,v_n)$ is strictly increasing. By construction, $\overline F(x)$ is the limit of the sequence $F(v_n)=m^f(x,y,z_n)$. By Proposition~\ref{invariance of straight}, there exists a straight line $\g'\cu Y$ with endpoints $f(x)$ and $f(y)$. Now, the fact that $c_n\ra+\infty$ implies that the points $m^f(x,y,z_n)\in\g'$ converge to $f(x)$. Hence $f(x)=\overline F(x)$.
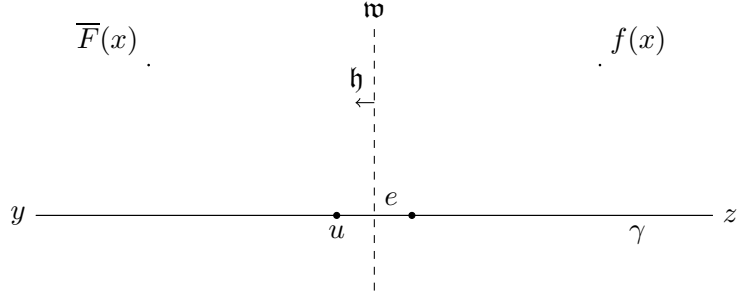
\begin{figure} 
\begin{tikzpicture}
\draw [fill] (-4.5,0) -- (4.5,0);
\node[left] at (-4.5,0) {$y$};
\node[right] at (4.5,0) {$z$};
\node[above right] at (0,0) {$e$};
\draw [dashed] (0,-1) -- (0,2.5);
\node[above] at (0,2.5) {$\mf{w}$};
\draw [->] (0,1.5) -- (-0.25, 1.5);
\node [above left] at (0,1.5) {$\mf{h}$};
\node[below] at (-0.5,0) {$u$};
\draw[fill] (-3,2) circle [radius=0.01cm];
\draw[fill] (3,2) circle [radius=0.01cm];
\node[above left] at (-3,2) {$\overline{F}(x)$};
\node [above right] at (3,2) {$f(x)$};
\draw[fill] (-0.5,0) circle [radius=0.04cm];
\draw[fill] (0.5,0) circle [radius=0.04cm];
\node[below] at (3.5,0) {$\g$};
\end{tikzpicture}
\caption{The case when $x\in\partial X\setminus\partial_sX$.}
\label{non-straight} 
\end{figure}

We are left to handle points $x\in\partial X\setminus\partial_sX$. Suppose for the sake of contradiction that $\overline F(x)$ and $f(x)$ are separated by a hyperplane $\mf{w}$. Consider an edge $e$ crossing $\mf{w}$ and extend $e$ to a straight line $\g$; let $u$ be the endpoint of $e$ on the same side of $\mf{w}$ as $\overline F(x)$. Name $y$ and $z$ the endpoints of $\g$ so that $\mf{w}\in\mscr{W}(\overline F(x),y|f(x),z)$. The situation is portrayed in Figure~\ref{non-straight}. We are going to construct a point $w\in\partial_sX$ such that $\min\{(\overline F(x)\cdot y)_u,(\overline F(x)\cdot w)_u\}<+\infty$ and $m(y,z,w)\in X$. We first show how to use $w$ to conclude the proof. 

First, observe that we have $(\overline F(x)\cdot z)_u=0$ and $(y\cdot z)_u=0$. Our choice of $w$ also implies that $(y\cdot w)_u$, $(z\cdot w)_u$ and at least one among $(\overline F(x)\cdot y)_u$ and $(\overline F(x)\cdot w)_u$ are finite, so $(\overline F(x),y,z,w)\in\mscr{A}(X)$. As $\overline F$ and $f$ coincide on the set ${\{y,z,w\}\cu\partial_sX}$, we have:
\begin{multline}
\nonumber
\Cr\left(\overline F(x),w,y,z\right)=\Cr\left(x,\overline F^{-1}(w),\overline F^{-1}(y),\overline F^{-1}(z)\right)= \\
 =\Cr\left(x,f^{-1}(w),f^{-1}(y),f^{-1}(z)\right)=\Cr\left(f(x),w,y,z\right).
\end{multline}
On the other hand, observe that we have $(\overline F(x)\cdot y)_u\geq(f(x)\cdot y)_u=0$ and $0=(\overline F(x)\cdot z)_u<(f(x)\cdot z)_u$. Hence $\Cr\left(\overline F(x),w,y,z\right)>\Cr\left(f(x),w,y,z\right)$,
a contradiction.
 
Regarding the construction of the point $w$, observe that $\g$ contains at least two vertices $v,v'$ of degree at least $3$; this is because $\overline F(x)$ and $f(x)$ are not straight and project to different points of $\g$. We can assume that $v'\in I(v,z)$. Let $e_y$ and $e_z$ be the only edges at $v$ that lie in $I(y,v)$ and $I(v,z)$, respectively; cf.\ Figure~\ref{construction of w}. Let moreover $\mc{E}$ be the set of edges $\eps$ at $v$ with $\mf{w}(\eps)\in\mscr{W}(v|\overline F(x))$. We distinguish three cases.
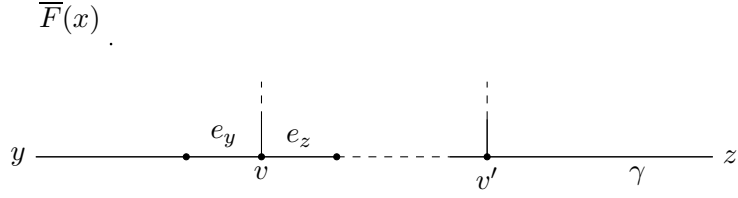
\begin{figure} 
\begin{tikzpicture}
\draw [fill] (-4.5,0) -- (-0.5,0);
\draw [fill] (1,0) -- (4.5,0);
\draw [dashed] (-0.5,0) -- (1,0);
\node[left] at (-4.5,0) {$y$};
\node[right] at (4.5,0) {$z$};
\draw[fill] (-3.5,1.5) circle [radius=0.01cm];
\node[above left] at (-3.5,1.5) {$\overline{F}(x)$};
\draw[fill] (-1.5,0) circle [radius=0.04cm];
\draw[fill] (-2.5,0) circle [radius=0.04cm];
\draw[fill] (-0.5,0) circle [radius=0.04cm];
\draw[fill] (1.5,0) circle [radius=0.04cm];
\node[below] at (-1.5,0) {$v$};
\node[below] at (1.5,0) {$v'$};
\node[below] at (3.5,0) {$\g$};
\node[above] at (-2,0) {$e_y$};
\node[above] at (-1,0) {$e_z$};
\draw [fill] (-1.5,0) -- (-1.5,0.5);
\draw [dashed] (-1.5,0.5) -- (-1.5,1);
\draw [fill] (1.5,0) -- (1.5,0.5);
\draw [dashed] (1.5,0.5) -- (1.5,1);
\end{tikzpicture}
\caption{The general setup for the construction of the point $w$.}
\label{construction of w} 
\end{figure}

{\bf Case~1:} either $(\overline F(x)\cdot y)_v<+\infty$ or $\#\mc{E}=1$. It suffices to pick any edge $\eps$ incident to $v$ and distinct from $e_y$ and $e_z$; extending $\eps$ to a straight ray we obtain $w\in\partial_sX$ with $m(y,z,w)=v$. If $(\overline F(x)\cdot y)_v<+\infty$, we are done. If $(\overline F(x)\cdot y)_v=+\infty$ and $\#\mc{E}=1$, we have $\mc{E}=\{e_y\}$; hence $(\overline F(x)\cdot w)_v=0$ by Lemma~\ref{straight rays 2}.

{\bf Case~2:} $(\overline F(x)\cdot y)_v=+\infty$ and no edge in $\mc{E}$ spans a square with $e_z$. Replacing $v$ with $v'$, we end up again in the situation where $\#\mc{E}=1$, which was handled in the previous case.

{\bf Case~3:} $(\overline F(x)\cdot y)_v=+\infty$, $\#\mc{E}\geq 2$ and some $\eps\in\mc{E}$ spans a square with $e_z$. Since $v$ is not an extremal vertex, there exists an edge $\eps'$ at $v$ that does not span a square with $\eps$. In particular, $\eps'\neq e_z$ and $\eps'\not\in\mc{E}$, which also ensures that $\eps'\neq e_y$. We extend $\eps'$ to a straight ray $r$ and set $w=r^+$. Lemma~\ref{straight rays 2} implies that $m(y,z,w)=v$ and $(\overline F(x)\cdot w)_v=0$. 
\end{proof}

\bibliography{mybib}
\bibliographystyle{alpha}

\end{document}